 \documentclass[noinfoline,11pt]{imsart}

\usepackage{graphicx}
\usepackage{comment}
\usepackage{url} 

\usepackage{hyperref}

\usepackage[left=24mm, right=24mm, top=30mm, bottom=30mm]{geometry}

\usepackage{amssymb}
\usepackage{amsthm}

\usepackage{mathrsfs}
\usepackage{mathtools}
\usepackage{amsmath}
\usepackage{amsfonts}

\newcommand{\hatkappal}{{\hat{\mathbf{K}}_\alpha}}
\newcommand{\kappal}{\mathbf{K}_{\alpha}}
\newcommand{\hilbert}{{\mathbb{H}}}

\newcommand{\hs}{\mathcal{S}_2}

\newtheorem{theorem}{Theorem}
\newtheorem{condition}{Condition}
\newtheorem{example}{Example}
\newtheorem{proposition}{Proposition}
\newtheorem{remark}{Remark}

\begin{document}

\begin{frontmatter}

\title{On the rate of convergence for the autocorrelation operator in functional autoregression\thanksref{T1}}
\runtitle{Autocorrelation operator in functional autoregression}
\thankstext{T1}{To appear in \emph{Statistics \& Probability Letters}}

\begin{aug}
\author{\fnms{Alessia} \snm{Caponera}\ead[label=e1]{alessia.caponera@epfl.ch}}
\and
\author{\fnms{Victor M.} \snm{Panaretos}\ead[label=e4]{victor.panaretos@epfl.ch}}

\runauthor{Caponera \& Panaretos}

\affiliation{\'Ecole Polytechnique F\'ed\'erale de Lausanne}

\address{\'Ecole Polytechnique F\'ed\'erale de Lausanne\\ \printead{e1,e4} }

\end{aug}

\begin{abstract} 
We consider the problem of estimating the autocorrelation operator of an autoregressive Hilbertian process. By means of a Tikhonov approach, we establish a general result that yields the convergence rate of the estimated autocorrelation operator as a function of the rate of convergence of the estimated lag zero and lag one autocovariance operators. The result is general in that it can accommodate any consistent estimators of the lagged autocovariances. Consequently it can be applied to processes under any mode of observation: complete, discrete, sparse, and/or with measurement errors. An appealing feature is that the result does not require delicate spectral decay assumptions on the autocovariances but instead rests on natural source conditions. The result is illustrated by application to important special cases. \end{abstract}

\bigskip
\begin{keyword}[class=AMS]
\kwd[Primary ]{62G08}
\kwd[; secondary ]{62M}
\end{keyword}

\begin{keyword}
\kwd{functional time series}
\kwd{source condition}
\kwd{Tikhonov regularization}
\end{keyword}

\end{frontmatter}

\setcounter{tocdepth}{1}
\tableofcontents

\section{Introduction}

Let $\hilbert$ be a real and separable Hilbert space with norm $\|\cdot\|_\hilbert$ and inner product $\langle \cdot, \cdot \rangle_\hilbert$. Let $\mathcal{L}(\hilbert)$ be the set of all bounded linear operators from $\hilbert$ to $\hilbert$, and $\|\cdot\|_\infty$ be the corresponding operator norm.

In this paper, we consider a collection of $\hilbert$-valued zero-mean random elements $\mathcal{X}=\{X_t, t \in \mathbb{Z}\}$, constituting the unique stationary solution of the autoregressive equation
\begin{equation*}
X_t = \rho(X_{t-1})  + \varepsilon_t,
\end{equation*}
with $\rho\in \mathcal{L}(\hilbert)$ such that $\sum_{j=0}^\infty \|\rho^j\|_\infty^2 < \infty$, and $\varepsilon=\{\varepsilon_t, t \in \mathbb{Z}\}$ a \emph{strong white noise} on $\hilbert$ -- see \cite[Definition 3.1]{bosq}. The collection $\mathcal{X}$ is called an autoregressive Hilbertian process of order one, written ARH(1). It plays a fundamental role in the field of functional time series. An extensive review of ARH(1) can be found in \cite{bosq}; while for a general discussion on central topics in functional data analysis see the surveys \cite{CUEVAS2014, WANG2016, GOIA2016, ANEIROS2019, ANEIROS2022}.


In this setting, a key role is played by the covariance and cross-crovariance operators (also known as lag 0 and lag 1 autocovariances), namely,
$$
\mathscr{R}_0 = \mathbb{E}[X_0 \otimes X_0], \qquad \mathscr{R}_1 = \mathbb{E}[X_1 \otimes X_0],
$$
where the tensor product $u \otimes v$, with $u,v \in \hilbert$, is defined to be the mapping that takes any element $f \in \hilbert$ to $u \langle f, v\rangle \in \hilbert.$ 
Note that $\mathscr{R}_0$ is a nonnegative definite, self-adjoint, trace-class operator and that the adjoint of $\mathscr{R}_1$ is given by $\mathscr{R}_1^*=\mathscr{R}_{-1}= \mathbb{E}[X_0 \otimes X_1]$. The two operators $(\mathscr{R}_0, \mathscr{R}_1)$ are inextricably linked with the operator $\rho$ through the  equations 
\begin{equation}\label{inverse-problem}
\mathscr{R}_1 = \rho \mathscr{R}_0, \qquad \mathscr{R}_{-1} =\mathscr{R}_0 \rho^*,
\end{equation}
which explain why the operator $\rho$ (or indeed its adjoint $\rho^*$) can be interpreted as the \emph{autocorrelation operator}.
In order to ensure identifiability, we shall assume that $\mathscr{R}_0$ is positive definite, i.e., $\operatorname{Ker}(\mathscr{R}_0)=\{0\}$ -- see for instance \cite{mas2007}. Otherwise, without loss of generality, one can take $\hilbert=\overline{\operatorname{Im}(\mathscr{R}_0)}$.


Estimating $\rho$ is fundamental to inference and prediction, but constitutes a subtle problem.
For instance, likelihood approaches are intractable in a truly infinite-dimensional framework. And, while  $\mathscr{R}_0$ and $\mathscr{R}_1$ are amenable to a plethora of estimation methods,  
plugging their estimators in \eqref{inverse-problem} and solving for $\rho$ must be done  carefully: this operation requires the definition of a pseudo-inverse operator, which is an unbounded operator, and whose domain is a strict subset of $\hilbert$ -- an ill-posed inverse problem.

For this reason, any  estimation procedure necessarily goes through the idea of regularizing an estimate of $\mathscr{R}_0$ to then obtain an approximation of its pseudo-inverse. 
The classical estimator based on a finite stretch $\{X_1,...,X_n\}$ of fully observed mean-zero elements was introduced by \cite{bosq91} and is a projection estimator. As the name suggests, it is constructed by first projecting $\{X_1,...,X_n\}$ onto the span $\hilbert_{k_n}$ of the first $k_n$ eigenvectors of the empirical covariance operator $$\hat{\mathscr{R}}_0 = \frac1n \sum_{i=1}^n X_t \otimes X_t.
$$ 
This yields a finite-dimensional process on $\hilbert_{k_n}$, and allows Equation \eqref{inverse-problem} to be restricted to a well-posed equation in $\hilbert_{k_n}$. The resulting estimator can be interpreted on all of $\hilbert$ as a finite rank operator with range $\hilbert_{k_n}$. Allowing $k_n$ to diverge at an appropriate rate (related to the spectral decay of $\mathscr{R}_0$) yields almost sure consistency \cite{bosq}, or asymptotic normality \cite{mas99}. \cite{guillas2001} obtained rates of convergence in the $L^2$-sense for a slightly modified estimator, both when the $\mathscr{R}_0$-eigenvectors are known and when they are estimated. Asymptotic properties for predictions are investigated in \cite{mas2007,antoniadis2003} and an application is given in \cite{carre2020mas}. {Results on consistent estimation and prediction for Banach-valued autoregressive processes are derived in \cite{ruiz2019}.}

Our purpose in this paper is to study \emph{general} rates of convergence in the sense that we do not restrict our argument to specific estimators $\hat{\mathscr{R}}_0$ and $\hat{\mathscr{R}}_1$ of $\mathscr{R}_0$ and $\mathscr{R}_1$, respectively. This makes the results extremely general and applicable to many different situations. Examples of its application are provided both for \emph{continuously/densely} sampled functional data and \emph{sparsely/noisily} sampled functional data. As a matter of fact, these are the first rates ever established in the latter regime. To do so, we regularize Equation \eqref{inverse-problem} by means of a Tikhonov approach, instead of a spectral truncation approach (Section \ref{sec:results}). In particular, we also avoid awkward spectral decay conditions, and instead focus on arguably more natural \emph{source conditions}.

\paragraph{Additional notation}  We denote the Moore-Penrose generalized inverse of an operator $\mathscr{T} \in \mathcal{L}(\hilbert)$  by $\mathscr{T}^\dag$ (see \cite[Definition 3.5.7]{Hsing}).
For $1\le p <\infty$, we let $\|\cdot\|_p$ denote the $p$-Schatten norm, defined as
$$\|\mathscr{T}\|_p = \left ( \sum_{j=1}^\infty \langle e_j, (\mathscr{T}^* \!\mathscr{T} )^{p/2} e_j \rangle_\hilbert\right )^{1/p}, \quad \mathscr{T} \text{ compact} \in \mathcal{L}(\hilbert),$$
for any complete orthonormal system (CONS) $\{e_j, j \in \mathbb{N}\}$.
Clearly, $\|\cdot\|_2, \|\cdot\|_1$ are the Hilbert-Schmidt and trace norms.
Moreover, we let $\mathcal{S}_p=\mathcal{S}_p(\hilbert)$ be the set of all compact operators $\mathscr{T} \in \mathcal{L}(\hilbert)$ such that $\|\mathscr{T}\|_p <\infty$ (this definition can also be extended to $0< p < 1$).  Recall that $\hs$ endowed with the Hilbert-Schmidt inner product is a Hilbert space.

\section{Results}\label{sec:results}

Without loss of generality, we will focus on the equation $\mathscr{R}_{-1} = \mathscr{R}_0 \rho^*$. We assume that $\rho \in \hs$ and $\mathscr{R}_0$ positive definite. 
We also assume to have consistent estimators $\hat{\mathscr{R}}_1$ and $\hat{\mathscr{R}}_0$ of $\mathscr{R}_1$ and $\mathscr{R}_0$, respectively. These can be arbitrary, as long as $\hat{\mathscr{R}}_{1}$ is in $\hs$ and $\hat{\mathscr{R}}_0$ is in $\mathcal{S}_p$, for some $1 \le p\le \infty$.
We take $\hat{\mathscr{R}}_{-1}= \hat{\mathscr{R}}^*_1$. In particular, we assume $\hat{\mathscr{R}}_0$ to be self-adjoint, but not necessarily nonnegative definite. 

Now, consider the set of all linear and bounded operators from $\hs$ to $\hs$, say $\mathcal{L}(\hs)$.
The operator composition $J:\hs \to \hs$,
$\Phi \stackrel{J}{\mapsto} \mathscr{R}_0 \Phi$ is linear and bounded, so belongs to $\mathcal{L}(\hs)$. The same holds for $\hat{J}:\hs \to \hs$, $\Phi \stackrel{\hat{J}}{\mapsto} \hat{\mathscr{R}}_0 \Phi$.

We define our estimator of $\rho^*$ as the solution of the following regularization problem
\begin{equation}\label{min:stochastic}
\min_{\Phi \in \hs} \| \hat{\mathscr{R}}_{-1} - \hat{\mathscr{R}}_0 \Phi \|^2_{2} + \alpha \|\Phi\|^2_{2}.
\end{equation}
By \cite[Theorem 6.2.1]{Hsing} or \cite[Proposition 7.3]{hanke2017}, for any $\alpha >0$, there exists a unique minimizer $\hat{\Phi}_\alpha \in \hs$ which can be expressed in terms of the operator $\hat{J}$. However, we can give a more explicit solution of \eqref{min:stochastic} which directly involves the use of $\hat{\mathscr{R}}_0$, that is,
$
\hat{\Phi}_{\alpha} = \hatkappal \hat{\mathscr{R}}_{-1},
$
with
$
\hatkappal = (\hat{\mathscr{R}}_0\hat{\mathscr{R}}_0 + \alpha I)^{-1} \hat{\mathscr{R}}_0,
$ and $I$ the identity operator on $\hilbert$. This can be deduced by observing that, given any CONS $\{e_j, j \in \mathbb{N}\}$ on $\hilbert$ and $\Phi \in \hs$,
\begin{equation*}
\| \hat{\mathscr{R}}_{-1} - \hat{\mathscr{R}}_0 \Phi \|^2_{2} + \alpha \|\Phi\|^2_{2} =\sum_{j=1}^\infty \left (\|(\hat{\mathscr{R}}_{-1} - \hat{\mathscr{R}}_0 \Phi)e_j\|^2_\hilbert + \alpha \|\Phi e_j\|^2_\hilbert \right)
\end{equation*}
and the $j$-th term in the sum is minimized at $\hat{\Phi}_\alpha e_j$, again by an application of \cite[Theorem 6.2.1]{Hsing}.
Similarly, $\Phi_{\alpha}=\kappal \mathscr{R}_{-1}$, with
$
\kappal = (\mathscr{R}_0\mathscr{R}_0 + \alpha I)^{-1} \mathscr{R}_0 
$, is the unique minimizer of
\begin{equation}\label{min:deterministic}
\min_{\Phi \in \hs} \| \mathscr{R}_{-1} - \mathscr{R}_0 \Phi \|^2_{2} + \alpha \|\Phi\|^2_{2}.
\end{equation}
We can characterize the operator $\rho^*$ as the limiting solution of \eqref{min:deterministic}.

\begin{proposition}\label{prop:solution}
There exists a unique $\Phi \in \hs$ satisfying $\mathscr{R}_{-1} = \mathscr{R}_0 \Phi$, and this is given by 
$(\mathscr{R}_0\mathscr{R}_0)^\dag \mathscr{R}_0 \mathscr{R}_{-1}$. Consequently, $\rho^*=(\mathscr{R}_0\mathscr{R}_0)^\dag \mathscr{R}_0 \mathscr{R}_{-1}$. Moreover,
\begin{equation}\label{deterministic}
\lim_{\alpha\to 0} \|\Phi_\alpha - \rho^* \|_{2} = 0.
\end{equation}
\end{proposition}
The following condition is key in quantifying how fast this last limit goes to zero, i.e. the rate of convergence for the deterministic counterpart of $\hat{\Phi}_\alpha$ to $\rho^*$. It essentially entails that, for any $f \in \hilbert$, $\mathscr{R}_{-1} f$ is in the range of $\mathscr{R}_0 \mathscr{R}_0$.
\begin{condition}[Source condition]\label{cond:source}
For a given $1\le p \le \infty$, there exists $w\in \mathcal{L}(\hilbert)$ and a positive constant $M$ such that $\mathscr{R}_{-1} = \mathscr{R}_0 \mathscr{R}_0 w$, $\|w\|_{p}\le M$.
\end{condition}

\begin{remark}
The source condition can be equivalently written as a condition on $\rho^*$, i.e. $\rho^* = \mathscr{R}_0w$, $\|w\|_{p}\le M$. Moreover, the operator $w^* \mathscr{R}_0$ is well defined on the whole Hilbert space $\hilbert$ and $\rho=w^* \mathscr{R}_0$. This provides insight on the \emph{regularity} of $\rho$: since $\mathscr{R}_0 \in \mathcal{S}_1$, then
\begin{itemize}
    \item if $p=\infty$ we know that $\rho \in \mathcal{S}_1$ -- see for instance \cite[Theorem 7.8 (c)]{schatten} -- which means that it is at least as regular as $\mathscr{R}_0$;
    
\item if $1 \le p < \infty$ we have $\rho \in \mathcal{S}_{p/(p+1)}$ -- see for instance \cite[Theorem 7.8 (b)]{schatten} -- meaning that it is more regular than $\mathscr{R}_0$, but no more than half a degree of regularity (the maximal regularity gap is at $p=1$, where $\rho$ is in $\mathcal{S}_{1/2}$, compared to $\mathscr{R}_0$ being in $\mathcal{S}_1$).
\end{itemize}
{For concreteness, if the operators $\rho$ and $\mathscr{R}_0$ are simultaneously diagonalizable with eigenvalues $\{\mu_j\}$ and $\{\lambda_j\}$, respectively, then this condition translates to eigenvalue decay rates. Specifically, if $1\le p \le \infty$, we have that $\{\mu_j/\lambda_j\}\in \ell_p,$ where $\ell_p$ denotes the space of $p$-summable sequences.}
The reader can compare the source condition and these conclusions with Assumption $\textbf{A}_1$ in \cite{mas2007}.
\end{remark}

We are now able to state our main result which consists in rates of convergence for the estimator of the autocorrelation $\rho$, as a function of the rates for the estimation of $\mathscr{R}_0$ and $\mathscr{R}_1$. The convergence is expressed in terms of a general $p$-norm $\|\cdot\|_p$. Of course, the most interesting cases are $p=1,2,\infty$, corresponding to trace, Hilbert-Schmidt and operator norms, respectively.
\begin{theorem}\label{th:rate}
Suppose that Condition \ref{cond:source} holds for a given $1\le p\le \infty$. Assume also that $\hat{\mathscr{R}_0}$ and $\hat{\mathscr{R}_1}$ are estimators of $\mathscr{R}_0$ and $\mathscr{R}_1$ such that $\|\hat{\mathscr{R}_0} - \mathscr{R}_0\|^2_{p}= O_\mathbb{P}\left ( \gamma_n\right), \|\hat{\mathscr{R}_1} - \mathscr{R}_1\|^2_{p} = O_\mathbb{P}\left ( \gamma_n\right)$. Then,
\begin{equation}\label{res1}
    \|\hat{\Phi}_{\alpha_n} - \rho^*\|^2_{p} = \|\hat{\Phi}^*_{\alpha_n} - \rho\|^2_{p} = O_\mathbb{P}\left ( \frac{\gamma_n}{\alpha_n^3} + \alpha_n \right).
\end{equation}
If in addition $\alpha_n \sim \gamma_n^{1/4}$, then we have 
\begin{equation}\label{res2}
    \|\hat{\Phi}_{\alpha_n} - \rho^*\|^2_{p} =\|\hat{\Phi}^*_{\alpha_n} - \rho\|^2_{p}= O_\mathbb{P}\left ( \gamma_n^{1/4} \right).
\end{equation}
\end{theorem}

Here, we considered $\hat{\mathscr{R}_0}$ and $\hat{\mathscr{R}_1}$ both having the same rate, which is usually the case. Otherwise, one can immediately derive the alternative version
$$
\|\hat{\Phi}_{\alpha_n} - \rho^*\|^2_{p} = O_\mathbb{P}\left ( \frac{\gamma_{1;n}}{\alpha_n} + \frac{\gamma_{0;n}}{\alpha_n^3} + \alpha_n \right),
$$
with $\gamma_{0;n}$ and $\gamma_{1;n}$ such that $\|\hat{\mathscr{R}_0} - \mathscr{R}_0\|^2_{p} = O_\mathbb{P}\left ( \gamma_{0;n}\right)$, $\|\hat{\mathscr{R}_1} - \mathscr{R}_1\|^2_{p} = O_\mathbb{P}\left ( \gamma_{1;n}\right)$.

 \begin{remark}Theorem \ref{th:rate} is expressed in terms of an upper bound in probability, i.e. \textit{Big-O in probability}, but we can think of other kinds of rates, depending on how $\hat{\mathscr{R}}_0$ and $\hat{\mathscr{R}}_1$ converge. Indeed if one has rates of convergence for $\hat{\mathscr{R}_0}$ and $\hat{\mathscr{R}_1}$ in a almost sure or mean-square sense, this reflects on the type of convergence of our autocorrelation estimator.
 \end{remark}
 
\begin{remark}
{Our techniques and results are intrinsically Hilbert-space based. However, we might expect that this general Hilbert structure could be exploited, for instance, to get some results (albeit of weaker nature) with Banach norms via the weak embedding approach introduced by \cite{ruiz2019}.}
\end{remark}
 
We now illustrate the use of our result in two important cases. We highlight that the second of these gives the first ever rates for autocorrelation estimation under a sparse/noisy sampling regime.
\begin{example}[Complete Observation]
In classical functional time series analysis, unbiased estimators of $\mathscr{R}_0$ and $\mathscr{R}_1$, based on a functional sample of size $n$, are given by
$$
\hat{\mathscr{R}}_0 = \frac1n \sum_{i=1}^n X_t \otimes X_t, \qquad \hat{\mathscr{R}_1} = \frac{1}{n-1} \sum_{t=1}^{n-1} X_{t+1} \otimes X_t,
$$
Under the assumption $\mathbb{E}\|X_0\|^4_\hilbert <\infty$, for these sample estimators we have a parametric rate in the mean-square sense, that is,
$$
\mathbb{E}\|\hat{\mathscr{R}_0} - \mathscr{R}_0\|^2_{2}= O(n^{-1}), \qquad \mathbb{E}\|\hat{\mathscr{R}_1} - \mathscr{R}_1\|^2_{2} = O(n^{-1}),
$$
see \cite[Theorem 4.1 and Theorem 4.7]{bosq}. Thus, $\gamma_n = n^{-1}$ and, if the source condition is satisfied for some $p\ge 2$, for $\alpha_n \sim n^{-1/4}$ we have
$$
\mathbb{E}\|\hat{\Phi}_{\alpha_n} - \rho^*\|^2_{p} = O(n^{-1/4}).
$$
Moreover, if $\|X_0\|$ is bounded, one has almost surely
$$
\|\hat{\mathscr{R}_0} - \mathscr{R}_0\|^2_{2}= O\left (\frac{\log n}{n}\right), \qquad \|\hat{\mathscr{R}_1} - \mathscr{R}_1\|^2_{2} = O\left (\frac{\log n}{n}\right), 
$$
see \cite[Corollary 4.1 and Theorem 4.8]{bosq}. Hence,
for $\alpha_n \sim (\log n/ n )^{1/4}$ we obtain
$$
\|\hat{\Phi}_{\alpha_n} - \rho^*\|^2_{p}  = O\left ( \left (\frac{\log n}{n}\right)^{1/4}\right) \quad \text{a.s.}.
$$
\end{example}

\begin{example}[Sparse/Noisy Observation] Consider the $r$-th order Sobolev space of periodic functions on $[0, 2\pi]$, say $\mathcal{H}_r([0,2\pi])$, and assume that $\{X_t, t \in \mathbb{Z}\}$ is a functional autoregressive process on $\hilbert=\mathcal{H}_r([0,2\pi])$, for some $r>1$.
When dealing with (possibly noisy) sparsely observed functional data \cite{yao:2005b}, i.e. a model of the form
$$
Y_{t,j} = X_t(U_{t,j}) + \epsilon_{t,j}, \qquad t=1,\dots,n, \ j=1,\dots,m,
$$
where $U_{ij}$ are sampling locations and $\epsilon_{ij}$ noise disturbances, one can define the estimators for $\mathscr{R}_0$ and $\mathscr{R}_1$ as the minimizer of a Tikhonov regularization problem -- see for instance \cite[Section 8.3]{Hsing}.
Under suitable assumptions, for these estimators 
\begin{equation*}
\|\hat{\mathscr{R}_0} - \mathscr{R}_0\|^2_{2}= O_\mathbb{P}\left (\gamma_n\right), \qquad \|\hat{\mathscr{R}_1} - \mathscr{R}_1\|^2_{2} = O_\mathbb{P}\left (\gamma_n\right), 
\end{equation*}
with $\gamma_n = (n m/\log n )^{-2r/(2r+1)} + n^{-1}$. Hence, if the source condition is satisfied for some $p\ge 2$, we have
$$
\|\hat{\Phi}_{\alpha_n} - \rho^*\|^2_{p}  = O_\mathbb{P}\left ( \left( \frac{\log n}{nm}\right)^\frac{r}{2(2r+1)} + n^{-1/4} \right).
$$
The result for $\hat{\mathscr{R}}_0$ and $\hat{\mathscr{R}}_1$ is established in \cite{CFSP:21}. In fact in the more general setting of stationary sequences on $\mathcal{H}_r(\mathbb{S}^d)$ -- $\mathbb{S}^d$ being the $d$-dimensional hypersphere -- it is shown that $\gamma_n =(n m/\log n )^{-2r/(2r+d)} + n^{-1}$, $r>d$. The case described above corresponds to $d=1$.
\end{example}

\section{Proofs of formal statements}

\begin{proof}[Proof of Proposition \ref{prop:solution}]
Recall that $\mathscr{R}_{-1} = \mathscr{R}_0 \rho^*$ and $\rho^* \in \hs$.
It is easy to show that $(\mathscr{R}_0\mathscr{R}_0)^\dag \mathscr{R}_0 \mathscr{R}_{-1} = \mathscr{R}_0^\dag \mathscr{R}_{-1}$ is a solution of $\mathscr{R}_{-1} = \mathscr{R}_0 \Phi, \ \Phi \in \hs$. Indeed, 
$
 \mathscr{R}_0 \mathscr{R}_0^\dag \mathscr{R}_{-1}  =\mathscr{R}_0 \mathscr{R}_0^\dag \mathscr{R}_0  \rho^* = \mathscr{R}_0  \rho^* = \mathscr{R}_{-1}.
$

We now prove that the solution is unique. Since $\mathscr{R}_{-1} = \mathscr{R}_0 \rho^*$, any possible solution $\Phi \in \hs$ satisfies
$
\mathscr{R}_0(\Phi - \rho^*)f = 0,$ for all $f \in \hilbert,
$
which means that $(\Phi - \rho^*)f \in \operatorname{Ker}(\mathscr{R}_0),$ for all $f \in \hilbert$. However, $\mathscr{R}_0$ is positive definite and hence $\operatorname{Ker}(\mathscr{R}_0)=\{0\}$. Thus, $\Phi = \rho^*$.

On the other hand, we can argue that $\mathscr{R}_{-1} \in \operatorname{Im}(J) \subseteq \operatorname{Dom}(J^\dag)$. 
This implies that $\rho^*=J^\dag \mathscr{R}_{-1}$.
Hence, from an application of \cite[Theorem 6.2.2]{Hsing}, we have the result in Equation \eqref{deterministic}.
\end{proof}

\begin{proof}[Proof of Theorem \ref{th:rate}]
Recall that $\rho^* = (\mathscr{R}_0\mathscr{R}_0)^\dag  \mathscr{R}_0 \mathscr{R}_{-1}$. Then, we can write 
\begin{equation}\label{eq:deco}
\hat{\Phi}_\alpha - \rho^* = \hatkappal (\hat{\mathscr{R}}_{-1} - \mathscr{R}_{-1}) + (\hatkappal - \kappal)\mathscr{R}_{-1} + (\kappal - (\mathscr{R}_0\mathscr{R}_0)^\dag \mathscr{R}_0) \mathscr{R}_{-1}.
\end{equation}

Now, define $\hat{\mathscr{R}}_{0; \alpha} = \hat{\mathscr{R}}_0\hat{\mathscr{R}}_0 + \alpha I$ and $\mathscr{R}_{0; \alpha} = \mathscr{R}_0\mathscr{R}_0 + \alpha I$. Thus,
\begin{align*}
    \hatkappal - \kappal &= \hat{\mathscr{R}}_{0; \alpha}^{-1} \hat{\mathscr{R}}_{0}  \pm \hat{\mathscr{R}}_{0; \alpha}^{-1} \mathscr{R}_{0}   - \mathscr{R}_{0; \alpha}^{-1} \mathscr{R}_{0} =\hat{\mathscr{R}}_{0; \alpha}^{-1}(\hat{\mathscr{R}}_{0}  - \mathscr{R}_{0}) + (\hat{\mathscr{R}}_{0; \alpha}^{-1} - \mathscr{R}_{0; \alpha}^{-1}) \mathscr{R}_{0}\\
                &=\hat{\mathscr{R}}_{0; \alpha}^{-1}(\hat{\mathscr{R}}_{0}  - \mathscr{R}_{0}) + \hat{\mathscr{R}}_{0; \alpha}^{-1}  (\mathscr{R}_{0; \alpha}  - \hat{\mathscr{R}}_{0; \alpha} ) \kappal.
\end{align*}
We first observe that, for $\alpha > 0 $ and $f \in \hilbert$,
\begin{align*}
    \alpha \|f\|_\hilbert^2 &\le \alpha \|f\|_\hilbert^2 + \|\mathscr{R}_0f\|_\hilbert^2 =  \langle f, (\mathscr{R}_0 \mathscr{R}_0 + \alpha I ) f \rangle_\hilbert \le \|f\|_\hilbert \|(\mathscr{R}_0 \mathscr{R}_0 + \alpha I ) f\|_\hilbert,
\end{align*}
and therefore
$$
\| \mathscr{R}_{0; \alpha}^{-1} \|_\infty := \sup_{\|f\|_\hilbert =1} \frac{1}{\|(\mathscr{R}_0 \mathscr{R}_0 + \alpha I ) f\|_\hilbert} \le \frac{1}{\alpha}.
$$
Equivalently, $ \| \hat{\mathscr{R}}_{0; \alpha}^{-1} \|_\infty \le 1/\alpha.
$

For the deterministic part in Equation \eqref{eq:deco}, we have
\begin{align*}
[ (\mathscr{R}_0\mathscr{R}_0)^\dag \mathscr{R}_0 - \kappal ] \mathscr{R}_{-1} &=  [ (\mathscr{R}_0\mathscr{R}_0)^\dag \mathscr{R}_0\mathscr{R}_0 - \mathscr{R}_{0; \alpha}^{-1} \mathscr{R}_0  \mathscr{R}_0 ] \rho^*\\
    &= \mathscr{R}_{0; \alpha}^{-1} [   \mathscr{R}_0\mathscr{R}_0 (\mathscr{R}_0\mathscr{R}_0)^\dag \mathscr{R}_0\mathscr{R}_0  + \alpha (\mathscr{R}_0\mathscr{R}_0)^\dag \mathscr{R}_0\mathscr{R}_0 - \mathscr{R}_0\mathscr{R}_0  ] \rho^* \\
&= \alpha\, \mathscr{R}_{0; \alpha}^{-1}  (\mathscr{R}_0\mathscr{R}_0)^\dag \mathscr{R}_0\mathscr{R}_0  \rho^* \\
&=\alpha\, \mathscr{R}_{0; \alpha}^{-1}  (\mathscr{R}_0\mathscr{R}_0)^\dag \mathscr{R}_0\mathscr{R}_0  (\mathscr{R}_0\mathscr{R}_0)^\dag \mathscr{R}_0 \mathscr{R}_{-1}\\
&=\alpha\, \mathscr{R}_{0; \alpha}^{-1}  \rho^*
\end{align*}
since $\mathscr{R}_0\mathscr{R}_0 (\mathscr{R}_0\mathscr{R}_0)^\dag \mathscr{R}_0\mathscr{R}_0 = \mathscr{R}_0\mathscr{R}_0 $ and $(\mathscr{R}_0\mathscr{R}_0)^\dag \mathscr{R}_0\mathscr{R}_0  (\mathscr{R}_0\mathscr{R}_0)^\dag = (\mathscr{R}_0\mathscr{R}_0)^\dag $.
In addition, for $z \in \hilbert$ and $f=(\mathscr{R}_{0}\mathscr{R}_{0} + \alpha I)^{-1}z$, we have
\begin{align*}
    \|\kappal^* z\|^2_{\hilbert}
    &= \langle \mathscr{R}_{0}(\mathscr{R}_{0}\mathscr{R}_{0} + \alpha I)^{-1} z, \mathscr{R}_{0} (\mathscr{R}_{0}\mathscr{R}_{0} + \alpha I)^{-1} z \rangle_{\hilbert}  = \langle f, \mathscr{R}_{0}\mathscr{R}_{0}f \rangle_{\hilbert}\\
    &\le \langle f, \mathscr{R}_{0}\mathscr{R}_{0}f \rangle_{\hilbert} +\alpha \langle f, f \rangle_{\hilbert}= \langle f, z \rangle_{\hilbert} \le \|f\|_{\hilbert} \|z\|_{\hilbert} \le \frac1\alpha \|z\|^2_{\hilbert},
\end{align*}
and hence
$
    \|\kappal\|_\infty = \|\kappal^*\|_\infty \le 1/\sqrt{\alpha}.
$
The same holds for $\hatkappal$.

Now, due to the source condition, we have $\rho^* = \mathscr{R}_0w$, $\|w\|_{p}\le M$, and
$$
\| [ (\mathscr{R}_0\mathscr{R}_0)^\dag \mathscr{R}_0 - \kappal ] \mathscr{R}_{-1} \|_{p} \le \alpha \|\kappal\|_\infty \|w\|_{p} \le \sqrt{\alpha} M.
$$
Thus,
\begin{align*}
\|\hat{\Phi}_\alpha - \rho^*\|_{p} &\le \| \hatkappal (\hat{\mathscr{R}}_{-1} - \mathscr{R}_{-1}) \|_{p} + \| (\hatkappal - \kappal)\mathscr{R}_{-1}  \|_{p} + \| (\kappal - (\mathscr{R}_0\mathscr{R}_0)^\dag \mathscr{R}_0) \mathscr{R}_{-1} \|_{p} \\
&\le \| \hatkappal  \|_\infty \|  \hat{\mathscr{R}}_1 - \mathscr{R}_1\|_{p} + \| \hat{\mathscr{R}}_{0; \alpha}^{-1}  \|_\infty \|  \hat{\mathscr{R}}_{0} - \mathscr{R}_{0}\|_{p} \|\mathscr{R}_1 \|_{p}  \\& +   \| \hat{\mathscr{R}}_{0; \alpha}^{-1} \|_\infty   \|  \hat{\mathscr{R}}_{0; \alpha} - \mathscr{R}_{0; \alpha}\|_{p}  \| \kappal \|_\infty \|\mathscr{R}_1\|_{p} + \sqrt{\alpha} M.
\end{align*}
Moreover, observe that
\begin{align*}
 \|  \hat{\mathscr{R}}_{0; \alpha} - \mathscr{R}_{0; \alpha}\|_{p} &=  \|  \hat{\mathscr{R}}_{0}\hat{\mathscr{R}}_{0} \pm \hat{\mathscr{R}}_{0}\mathscr{R}_{0} - \mathscr{R}_{0}\mathscr{R}_{0}\|_{p} \le \|  \hat{\mathscr{R}}_{0}\|_{p} \| \hat{\mathscr{R}}_{0}  - \mathscr{R}_{0}\|_{p} + \| \hat{\mathscr{R}}_{0}  - \mathscr{R}_{0}\|_{p} \|  \mathscr{R}_{0}\|_{p}\\
  &\le \| \hat{\mathscr{R}}_{0}  - \mathscr{R}_{0}\|^2_{p} + 2\| \hat{\mathscr{R}}_{0}  - \mathscr{R}_{0}\|_{p} \|  \mathscr{R}_{0}\|_{p} = O_\mathbb{P}\left ( \| \hat{\mathscr{R}}_{0}  - \mathscr{R}_{0}\|_{p} \right) 
\end{align*}

Thus, if $\|\hat{\mathscr{R}_0} - \mathscr{R}_0\|^2_{p}= O_\mathbb{P}\left ( \gamma_n\right), \|\hat{\mathscr{R}_1} - \mathscr{R}_1\|^2_{p} = O_\mathbb{P}\left ( \gamma_n\right)$, we can obtain the result in Equation \eqref{res1}.
If in addition $\alpha_n = \gamma_n^{1/4}$, then we have \eqref{res2}.
\end{proof}

\bibliographystyle{imsart-number}
\bibliography{mybiblio}
\end{document}